\RequirePackage[leqno]{amsmath}
\documentclass[leqno]{amsart}
\usepackage{amsmath}
\makeatletter
\usepackage[utf8]{inputenc}     
\usepackage[T1]{fontenc}  
\usepackage{multicol,multirow}
\usepackage{color, graphics}
\usepackage{pgfplots}
\usepackage[output-decimal-marker={,},exponent-product=\cdot]{siunitx}
\usepackage{graphicx, multicol, latexsym, amsmath, amssymb}
\usepackage{amsfonts}
\usepackage{enumitem}
\setlist[itemize]{topsep=0pt,after=\vspace{1.5\baselineskip}}
 \usepackage[colorlinks=true]{hyperref}
\usepackage{empheq}
\usepackage[T1]{fontenc}
\usepackage{tabularx}
\usepackage[leftcaption]{sidecap}
\sidecaptionvpos{figure}{m}
\usepackage{tikz}
\usepackage{subfigure}
\usepackage[margin=0.5in]{geometry}
\usetikzlibrary{backgrounds,automata}

\newcount\rc@count
\let\rc@clearconstantlist\empty
\newcommand\rc@clearconstant[1]{\global\expandafter\let\csname rc@const@#1\endcsname\undefined}
\newcommand\resetconstants[1]{%
    \def\rc@constname{#1}
    \global\rc@count=1\relax 
    \bgroup 
        \let\\\rc@clearconstant 
        \rc@clearconstantlist
        \global\let\rc@clearconstantlist\empty 
    \egroup
}
\newcommand\const[1]{%
    \@ifundefined{rc@const@#1}{%
        \expandafter\xdef\csname rc@const@#1\endcsname{%
           \noexpand\rc@useconst{\rc@constname}{\the\rc@count}%
        }%
        \g@addto@macro\rc@clearconstantlist{\\{\mathrm{#1}}}%
        \global\advance\rc@count1\relax
    }{}%
    \csname rc@const@#1\endcsname
}
\newcommand\rc@useconst[2]{{#1}\textsubscript{#2}}
\setlist[itemize]{noitemsep, topsep=0pt}

\def\R{\mathbb R}

\def\R{\mathbb R}  
\def\TM{T_{max}} 
\def
\@cite
#1#2{[{{\bfseries #1}\if@tempswa , #2\fi}]}
\makeatother
\newtheorem{theorem}{Theorem}[section]

\newtheorem{lemma}[theorem]{Lemma}

\newtheorem{remark}{Remark}
\newtheorem{algorithm}{Algorithm}

\newcounter{cnstcnt}

\title[Boundedness in an attraction-repulsion chemotaxis system with nonlinear productions] 
      {
Some refined criteria toward boundedness in an attraction-repulsion chemotaxis system with nonlinear productions}
\author[A. Columbu, S. Frassu and G. Viglialoro]{}
\subjclass[2010]{Primary: 35A01, 35K55, 35Q92. Secondary:  92C17.}
\keywords{Chemotaxis, Global existence, Boundedness, Nonlinear production. \\
	\textit{$^*$Corresponding author}: giuseppe.viglialoro@unica.it}
\begin{document}
\maketitle

\centerline{\scshape Alessandro Columbu, Silvia Frassu \and Giuseppe Viglialoro$^{*}$}
\medskip
{
 \medskip
 \centerline{Dipartimento di Matematica e Informatica}
 \centerline{Universit\`{a} di Cagliari}
 \centerline{Via Ospedale 72, 09124. Cagliari (Italy)}
 \medskip
}
\bigskip
\begin{abstract}
We study some zero-flux attraction-repulsion chemotaxis models,  with nonlinear production rates for the chemorepellent and the chemoattractant, whose formulation can be schematized as:
\begin{equation}\label{problem_abstract}
\tag{$\Diamond$}
\begin{cases}
u_t= \Delta u - \chi \nabla \cdot (u \nabla v)+\xi \nabla \cdot (u \nabla w)  & \textrm{ in } \Omega \times (0,T_{max}),\\
\tau v_t=\Delta v-\varphi(t,v)+f(u)  & \textrm{ in } \Omega \times (0,T_{max}),\\
\tau w= \Delta w - \psi(t,w) + g(u)& \textrm{ in } \Omega \times (0,T_{max}).
\end{cases}
\end{equation}
In this problem, $\Omega$ is a bounded and smooth domain of $\R^n$, for $n\geq 2$, $\chi,\xi>0$, $f(u)$, $g(u)$ reasonably regular functions generalizing, respectively, the prototypes $f(u)=\alpha u^k$ and $g(u)= \gamma u^l$, for some $k,l,\alpha,\gamma>0$ and all $u\geq 0$. Moreover, $\varphi(t,v)$ and  $\psi(t,w)$ have specific expressions, $\tau\in\{0,1\}$ and $\Theta:=\chi \alpha-\xi\gamma.$  Once for any sufficiently smooth $u(x,0)=u_0(x)\geq 0$, $\tau v(x,0)=\tau v_0(x)\geq 0$ and $\tau w(x,0)=\tau w_0(x)\geq 0$ the local well posedness of problem \eqref{problem_abstract} is ensured, we establish for the classical solution $(u,v,w)$ defined in $\Omega \times (0,T_{max})$ that the life span is indeed $\TM=\infty$ and $u, v$ and $w$ are uniformly bounded in $\Omega\times (0,\infty)$ in the following cases:
\begin{enumerate}[label=\Roman*),ref=\Roman*]
\item For $\varphi(t,v)=\beta v$, $\beta>0$,  $\psi(t,w)=\delta w$, $\delta>0$ and $\tau=0$, provided 
\begin{enumerate}[label=\theenumi.\arabic*),ref=\theenumi.\arabic*]
\item  $k<l$;
\item $k,l\in\left(0,\frac{2}{n}\right)$;
\item \label{AbstractAss} $k=l$ and $\Theta<0$, or 
 $l=k\in\left(0,\frac{2}{n}\right)$ and $\Theta\geq 0.$ 
\end{enumerate}
\item For $\varphi(t,v)=\beta v$, $\beta>0$,  $\psi(t,w)=\delta w$, $\delta>0$ and 
$\tau=1$, whenever 
\begin{enumerate}[label=\theenumi.\arabic*),ref=\theenumi.\arabic*]
\item $l,k\in \left(0,\frac{1}{n}\right]$;
\item  $l\in \left(\frac{1}{n},\frac{1}{n}+\frac{2}{n^2+4}\right)$ and $k\in \left(0,\frac{1}{n}\right]$, or  $k\in \left(\frac{1}{n},\frac{1}{n}+\frac{2}{n^2+4}\right)$ and $l\in \left(0,\frac{1}{n}\right]$;
\item  $l,k\in \left(\frac{1}{n},\frac{1}{n}+\frac{2}{n^2+4}\right).$
\end{enumerate} 
\item \label{OpenQuestionAbstractSolved} For $\varphi(t,v)=\frac{1}{|\Omega|}\int_\Omega f(u)$ and $\psi(t,w)=\frac{1}{|\Omega|}\int_\Omega g(u)$ and $\tau=0$, under the assumptions $k<l$ or \ref{AbstractAss}).
\end{enumerate}
In particular, in this paper we partially improve what derived in \cite{ViglialoroMatNacAttr-Repul} and solve an open question given in \cite{LiuLi-2021NLARWA}. Finally, the research is complemented with numerical simulations in bi-dimensional domains. 
\end{abstract}
\resetconstants{c}
\section{Introduction and motivations}
\subsection{Presentation of the model}\label{IntroSection}
This article is mainly dedicated to the following Cauchy boundary problem
\begin{equation}\label{problem}
\begin{cases}
u_t= \Delta u - \chi \nabla \cdot (u \nabla v)+\xi \nabla \cdot (u \nabla w)  & \textrm{ in } \Omega \times (0,T_{max}),\\
\tau v_t=\Delta v-\beta v +f(u)  & \textrm{ in } \Omega \times (0,T_{max}),\\
\tau w_t= \Delta w - \delta w + g(u)& \textrm{ in } \Omega \times (0,T_{max}),\\
u_{\nu}=v_{\nu}=w_{\nu}=0 & \textrm{ on } \partial \Omega \times (0,T_{max}),\\
u(x,0)=u_0(x), \; \tau v(x,0)=\tau v_0(x), \; \tau v(x,0)=\tau v_0(x) & x \in \bar\Omega,
\end{cases}
\end{equation}
defined  in a bounded and smooth domain $\Omega$ of $\R^n$, with $n\geq 2$, $\chi, \xi, \beta,\delta>0$, $\tau\in\{0,1\}$ and some functions $f=f(s)$ and $g=g(s)$, sufficiently regular in their argument $s\geq 0$, and essentially behaving as $s^k$ and $s^l$, for some $k,l>0$, and further regular initial data $u_0(x),\tau v_0(x),\tau w_0(x)\geq 0$.  Moreover,  $u_\nu$ (and similarly for $v$ and $w$) indicates the outward normal derivative of $u$ on $\partial \Omega$, whereas $T_{max}$ identifies the maximum time up to which solutions to the system can be extended. 

In conformity with the celebrated Keller--Segel models (\cite{K-S-1970,Keller-1971-MC,Keller-1971-TBC}), if $u=u(x,t)$ stands for a certain population density (bacteria, organisms, cells)  at the 
position $x$ and at the time $t$, and $v=v(x,t)$ and $w=w(x,t)$ denote, respectively,  the concentration of an attractive and repulsive chemical signal (chemoattractant and chemorepellent),  model \eqref{problem} idealizes the natural mechanism in which the 
automatic diffusion of the cells (the Laplacian $\Delta u$), initially distributed according to $u_0$, is perturbed by the aggregation/repulsion impact from the cross terms $\chi u \nabla v/\xi u \nabla w$ (increasing for larger values of $\chi$ and $\xi$). In addition,  the initial distribution $\tau v_0$ (respectively, $\tau w_0$) of the chemoattractant (respectively, chemorepellent)  diffuses and cells contribute to its growth in agreement to the law $f(u)$ (respectively, $g(u)$); the specified dynamics occurs in a totally impenetrable domain (homogeneous Neumann boundary conditions).
 \subsection{A view on the state of the art} 
In the context of the mentioned Keller--Segel model with single productive signal, problem \eqref{problem} is a combination of this aggregative signal-production mechanism  
\begin{equation}\label{problemOriginalKS} 
u_t= \Delta u - \chi  \nabla \cdot (u \nabla v) \quad \textrm{and} \quad 
\tau v_t=\Delta v-v+u,  \quad  \textrm{ in } \Omega \times (0,T_{max}),
\end{equation}
and this repulsive signal-production one 
\begin{equation}\label{problemOriginalKSCosnumption}
u_t= \Delta u+ \xi \nabla \cdot (u \nabla w)  \quad \textrm{and} \quad 
\tau w_t=\Delta v-u +w,  \quad  \textrm{ in } \Omega \times (0,T_{max}).
\end{equation}
Let us discuss some details on these problems and, since we will deal with classical solutions, we confine our attention to selected references; in this sense, we precise that the literature is richer. 

From the one hand,  if problem \eqref{problemOriginalKS} is considered, since the attractive signal $v$ increases with $u$, the inertial homogenization process of the cells' density could interrupt and very high and spatially concentrated spikes formations (\textit{chemotactic collapse}) may appear; this is especially due to the size of $\chi$, the initial mass of the particle distribution, i.e.,  $m=\int_\Omega u_0(x)dx,$ and the space dimension. Indeed, if in the one-dimensional setting blow-up phenomena are excluded  (see \cite{OsYagUnidim}), in higher dimensions if $m \chi$  surpasses a certain critical value $m_\chi$, the system might present the aforementioned chemotactic collapse, whereas for $m\chi<m_\chi$ no instability appears in the motion of the cells. There are many contributions dedicated to understanding this scenario. In this regard, in \cite{HerreroVelazquez,JaLu,Nagai,WinklAggre} (and references therein cited), the interested reader can find pointers to the rich literature dealing with the existence and properties of global, uniformly bounded or blow-up (local) solutions to the Cauchy problem associated to \eqref{problemOriginalKS}. On the other hand, as far as nonlinear segregation chemotaxis models like those we are considering, when in problem \eqref{problemOriginalKS}  the production $f(u)=u$ is replaced by $f(u)\cong u^k$, with  $0<k<\frac{2}{n}$ ($n\geq1$), uniform boundedness of all its solutions is proved in \cite{LiuTaoFullyParNonlinearProd}. Moreover, by resorting to a simplified parabolic-elliptic version in spatially radial contexts, when the second equation is reduced to $0=\Delta v-\frac{1}{|\Omega|}\int_\Omega f(u)+f(u)$, with $f(u)\cong u^k$, it is known  (see \cite{WinklerNoNLinearanalysisSublinearProduction}) that  the same conclusion on the boundedness continues to be valid for any $n\geq 1$ and $0<k<\frac{2}{n}$, whereas for $k>\frac{2}{n}$ blow-up phenomena may occur. 

Indeed, from the other hand, the literature on model \eqref{problemOriginalKSCosnumption} is rather  sparse and general (see, for instance, \cite{Mock74SIAM,Mock75JMAA} for analyses on similar contexts). In particular, no result on the blow-up scenario is available, and this is conceivable due to the repulsive nature of the mechanism.

Oppositely, quite different is the level of knowledge for attraction-repulsion chemotaxis problems accounting both \eqref{problemOriginalKS} and \eqref{problemOriginalKSCosnumption}; more exactly, in order to bring together precise results that can be compared each other, or at least discussed in a common spirit,  we point out that \textit{herein we prefer to focus on models exactly formulated as in \eqref{problem}}, otherwise the discussion might become too dispersive (see further comments in Remark \ref{RemarkOtherResults}). In the specific, we will consider not only analyses connected to the parabolic-elliptic-elliptic and the fully parabolic case of model \eqref{problem}, but also to a nonlocal version of the model itself. 
\subsection*{The case $\tau=0$} For linear growths of the chemoattractant and the chemorepellent, $f(u)=\alpha u$, $\alpha>0$, and $g(u)=\gamma u$, $\gamma>0$, 
we have that $\Theta:=\chi \alpha-\xi\gamma$, value of the difference between the attraction and repulsion impacts, is such that whenever $\Theta<0$ (repulsion-dominated regime), in any dimension all solutions to the model are globally bounded, whereas for $\Theta>0$ (attraction-dominated regime) and $n=2$  unbounded solutions can be detected (see \cite{GuoJiangZhengAttr-Rep,LI-LiAttrRepuls,TaoWanM3ASAttrRep,VIGLIALORO-JMAA-BlowUp-Attr-Rep,YUGUOZHENG-Attr-Repul} for some details on the issue). Indeed, for more general expressions of the production laws, respectively $f$ and $g$ generalizing the prototypes $f(u)=\alpha u^k$, $k>0,$ and $g(u)=\gamma u^l$, $l>0$, as far as we know the following are the most recent results, valid for $n\geq 2$ and derived in \cite{ViglialoroMatNacAttr-Repul}: For every  $\alpha,\beta,\gamma,\delta,\chi>0$, and $l>k\geq 1$ (resp. $k>l\geq 1$), there exists $\xi^*>0$ (resp. $\xi_*>0$) such that if $\xi>\xi^*$ (resp. $\xi\geq \xi_*$), any sufficiently regular initial datum $u_0(x)\geq 0$ (resp. $u_0(x)\geq 0$ small in some Lebesgue space) emanates a unique classical solution, uniformly bounded in time. Moreover, the same conclusion continues valid for any $\alpha,\beta,\gamma,\delta,\chi,\xi>0$, any sufficiently regular $u_0(x)\geq 0$ if the hypotheses $0<k<1$ and $l=1$ are accomplished. 
\subsection*{The case $\tau=1$} First, we point out that in the context of real applications model \eqref{problem}  was introduced in \cite{Luca2003Alzheimer} for one-dimensional settings and linear proliferation, to describe the aggregation of microglia observed in Alzheimer's disease. 

Moreover, from the stricter mathematical point of view, in \cite{TaoWanM3ASAttrRep} it is proved that in two-dimensional domains, for $f(u)=\alpha u$ and $g(u)=\gamma u$, sufficiently smooth initial data generate global-in-time bounded solutions whenever (recall $\Theta=\chi \alpha-\xi \gamma$) 
$$\Theta<0 \quad \textrm{and} \quad \beta=\delta  \quad \textrm{or} \quad  \Theta<0 \quad \textrm{and} \quad  -\frac{\chi^2\alpha^2(\beta-\delta)^2}{2\Theta \beta^2 C} \int_\Omega u_0(x)\leq 1, \quad \textrm{for some }\; C>0.$$
\subsection*{The nonlocal case} If we take into account this nonlocal version of model \eqref{problem}
\begin{equation}\label{problemConFunzioniTempo}
\begin{cases}
u_t= \Delta u - \chi \nabla \cdot (u \nabla v)+\xi \nabla \cdot (u \nabla w)  & \textrm{ in } \Omega \times (0,T_{max}),\\
0=\Delta v-\frac{1}{|\Omega|}\int_\Omega f(u) +f(u)  & \textrm{ in } \Omega \times (0,T_{max}),\\
0= \Delta w -\frac{1}{|\Omega|}\int_\Omega g(u) + g(u)& \textrm{ in } \Omega \times (0,T_{max}),\\
u_{\nu}=v_{\nu}=w_{\nu}=0 & \textrm{ on } \partial \Omega \times (0,T_{max}),\\
u(x,0)=u_0(x) & x \in \bar\Omega,\\
\int_\Omega v(x,t)dx=\int_\Omega w(x,t)dx=0 & t\in  (0,T_{max}),
\end{cases}
\end{equation}
in \cite{LiuLi-2021NLARWA} it is established that if $k>l$ and $k>\frac{2}{n}$, there are initial data producing solutions blowing up at some finite time while, for any $l>0$,
if $0<k < \frac{2}{n}$ such a phenomenon is prevented and all solutions are globally bounded in time. For the sake of clarity, in this model $v$ represents the \textit{deviation} of the chemical signal (and not the chemical substance itself, as previously indicated). In particular, since the deviation measures the difference between the signal concentration and its mean, rigorously in the model we would substitute $v$ with $\bar{v}$, being $\bar{v}=v-\frac{1}{|\Omega|}\int_\Omega v$. We avoid this relabeling, but we observe that conversely to what happens to the cell and chemical densities (which are nonnegative) the \textit{deviation} changes sign. In particular, from its definition, suddenly it is seen that its mean is zero (as fixed in the last assumptions of model \eqref{problemConFunzioniTempo}), which in turn ensures the uniqueness of the solution for the Poisson equation under homogeneous Neumann boundary conditions. Exactly the same explanation is valid for $w$.  
\begin{remark}[Studies in close contexts to model \eqref{problem}]\label{RemarkOtherResults}
The reason why we are putting our attention to models where diffusion and both sensitivities are linear and no logistic source perturbs the overall dynamics of the cells's density (as in the first equation of \eqref{problem}), is due to the consideration that when this is replaced with
\begin{equation}\label{NonLinearDiffSentLogistic-FirstEq}
u_t=\nabla (\cdot D(u)\nabla u -S(u)\nabla v+T(u)\nabla w)+h(u),
\end{equation}
for some $D(u), S(u), T(u)$ and $h(u)$, the overall mathematical analysis is in some sense more manageable. Indeed, for standard situations where $D(u)\simeq u^{m_1}, S(u)\simeq u^{m_2}, T(u)\simeq u^{m_3}$ and $h(u)\simeq \lambda u- \mu u^\vartheta$ (where $m_1,m_2,m_3,\lambda,\mu,\vartheta$ assume some real values), the technical machinery involving functional and algebraic inequalities is more inclined to work exactly due to the leeway given by those parameters. Naturally, this room for maneuver is much more reduced when the parameters are fixed and the specific model is not always recoverable from the more general one; moreover no direct comparison between the two formulations is strictly possible.

Despite that, we herein give hints to the most recent outcomes obtained in models close to \eqref{problem}. For linear diffusion and sensitivities, linear and nonlinear productions and logistic terms, criteria toward boundedness, long time behaviors and  blow-up issues for related solutions to \eqref{problem} in the case $\tau=0$, are studied in \cite{LiangEtAlAtt-RepNonLinProdLogist-2020,XinluEtAl2022-Asymp-AttRepNonlinProd,ChiyoMarrasTanakaYokota2021}. Similar analyses, in the case $\tau=1$, are discussed in \cite{GuoqiangBin-2022-3DAttRep}. Finally, moving to scenario where the particles' density obeys an equation as \eqref{NonLinearDiffSentLogistic-FirstEq}, with $D,S,T,h$ behaving as specified, for both the cases $\tau=0$ and $\tau=1$ for the equation of $v$ and $w$ we suggest the interesting paper \cite{GuoqiangBinATT-RepNonlinDiffSensLogistic}.   
\end{remark}
\section{Claims of the Theorems}
In this research we aim at extending the mathematical comprehension of attraction-repulsion Keller--Segel systems with linear diffusion and sensitivity and without logistics, by giving answers to questions not yet covered in the literature or improving already established results. 

To this scope, once these hypotheses are fixed
\begin{equation}\label{f}
f,g \in C^1([0,\infty)) \quad \textrm{with} \quad   0\leq f(s)\leq \alpha s^k  \textrm{ and } \gamma_0 (1+s)^l\leq g(s)\leq \gamma_1 (1+s)^l,\quad  \textrm{for some}\; \alpha, k, l >0, \gamma_1\geq\gamma_0>0,
\end{equation}
what follows is shown.
\begin{theorem}\label{MainTheorem}
Let $\Omega$ be a smooth and bounded domain of $\mathbb{R}^n$, with $n\geq 2$, $\tau=0$ and $\chi, \xi,\beta, \delta$ positive. Moreover, for some $\alpha,\gamma_0,\gamma_1>0$, let $f$ and $g$ comply with assumptions \eqref{f}
in at least one of these situations:
\begin{enumerate}
\item \label{itemkminltheoremelliptic} $k<l$;
\item $k,l\in\left(0,\frac{2}{n}\right)$;
\item \label{itemkUGUALEltheoremelliptic} $k=l$ and $\Theta_0:=\chi \alpha - \xi \gamma_0<0$, or $l=k\in\left(0,\frac{2}{n}\right)$ and $\Theta_0 \geq 0$. 
\end{enumerate}
Then for any initial data $u_0\in W^{1,\infty}(\bar{\Omega})$, with $u_0\geq 0$ on 
$\bar{\Omega}$, problem \eqref{problem} admits a unique global and uniformly bounded classical solution.  
\end{theorem}
\begin{theorem}\label{Main1Theorem}
Let $\Omega$ be a smooth and bounded domain of $\mathbb{R}^n$, with $n\geq 2$, $\tau=1$ and $\chi, \xi,\beta, \delta$ positive. 
Moreover, for some $\alpha,\gamma_0,\gamma_1>0$, let $f$ and $g$ comply with assumptions \eqref{f} in at least one of these situations:
\begin{enumerate}
\item \label{Item1-Parabol} $l,k\in \left(0,\frac{1}{n}\right]$;
\item  $l\in \left(\frac{1}{n},\frac{1}{n}+\frac{2}{n^2+4}\right)$ and $k\in \left(0,\frac{1}{n}\right]$, or  $k\in \left(\frac{1}{n},\frac{1}{n}+\frac{2}{n^2+4}\right)$ and $l\in \left(0,\frac{1}{n}\right]$;
\item  $l,k\in \left(\frac{1}{n},\frac{1}{n}+\frac{2}{n^2+4}\right).$
\end{enumerate} 
Then for any initial data  $(u_0, v_0,  w_0)\in (W^{1,\infty}(\Omega))^3$, with $u_0, v_0,w_0\geq 0$ on $\bar{\Omega}$, problem \eqref{problem} admits a unique global and uniformly bounded classical solution.  
\end{theorem}
\begin{theorem}\label{MainTheorem3}
Let $\Omega$ be a smooth and bounded domain of $\mathbb{R}^n$, with $n\geq 2$, and $\chi, \xi$ positive. Moreover, for some $\alpha,\gamma_0,\gamma_1>0$, let $f$ and $g$ comply with hypotheses \eqref{f}, for $k<l$ or for $k=l$, provided that $\Theta_0:=\chi \alpha - \xi \gamma_0$ complies with either $\Theta_0<0$, or 
$\Theta_0 \geq 0$ and $l=k\in\left(0,\frac{2}{n}\right)$. Then for any initial data $u_0\in W^{1,\infty}(\bar{\Omega})$, with $u_0\geq 0$ on $\bar{\Omega}$, problem \eqref{problemConFunzioniTempo} admits a unique global and uniformly bounded classical solution. 
\end{theorem}
\begin{remark}[Solving an open question and improvements of known results]
Let us give these comments:
\begin{itemize}
\item [$\rhd$] After \cite[Theorem 1.2]{LiuLi-2021NLARWA}, precisely in \cite[Remark 1.2]{LiuLi-2021NLARWA}, the authors leave an open question by textually writing:  ``we still have no ideas for the behavior of the solution in the case $k > \frac{2}{n}$ and $k<l$.''. The statement of Theorem \ref{MainTheorem3} provides the answer to this question and these two theorems give a more complete picture on the behavior of solutions to model \eqref{problemConFunzioniTempo}.
\item [$\rhd$] The condition $l>k>0$ derived in Theorem \ref{MainTheorem}, and yielding boundedness to model \eqref{problem}, improves that in \cite[Theorem 4.4]{ViglialoroMatNacAttr-Repul}.  
\end{itemize}
\end{remark}
\begin{remark}\label{RemarkUnifBounCalsSol}
 As usual in the nomenclature, in chemotaxis models a global and uniformly bounded classical solution to problem \eqref{problem} is a triplet  of nonnegative functions 
\begin{align}\label{ClassicalAndGlobability}
	\begin{split}
		u, v, w\in C^0(\bar{\Omega}\times [0,\infty))\cap  C^{2,1}(\bar{\Omega}\times (0,\infty)) 
		\cap L^\infty((0, \infty);L^{\infty}(\Omega)).
	\end{split}
\end{align}
\end{remark}
The rest of the paper is structured as follows. First, in $\S$\ref{LocalSol}  we recall some hints concerning the local existence and uniqueness of a classical solution to model \eqref{problem} and some of its main properties. In this same section we give a \textit{boundedness criterion}, establishing globality and boundedness of local solutions from proper \textit{a priori} $L^p$-boundedness, which is achieved in $\S$\ref{EstimatesAndProofSection} by relying on preliminaries collected in $\S$\ref{PreliminariesSection}. In the same $\S$\ref{EstimatesAndProofSection} we can deduce the claims of Theorems \ref{MainTheorem}, \ref{Main1Theorem} and \ref{MainTheorem3}.  Finally,
the theoretical results presented here are numerically investigated in planar domains in $\S$\ref{SectionSimulations}. 
\section{Existence of local-in-time solutions and main properties}\label{LocalSol}
\textit{From now on we will tacitly assume that all the appearing constants below $c_i$, $i=1,2,\ldots,$ are positive.} 

Once $\Omega$, $\chi,\xi$ and $f, g$ are fixed and obey what hypothesized above, $(u, v, w)$ will stand for the classical and nonnegative solution to problem \eqref{problem} (or \eqref{problemConFunzioniTempo}), defined for all $(x,t) \in \bar{\Omega}\times [0,T_{max})$, for some finite $T_{max}$, and emanating from nonnegative initial data $u_0, \tau v_0,  \tau w_0 \in W^{1,\infty}(\Omega)$. In particular, $u$, $v$ and $w$ are such that 
\begin{equation}\label{massConservation}
	\hspace{-0.1cm}
	\int_\Omega u(x, t)dx = \int_\Omega u_0(x)dx= m \textrm{ on }  (0,\TM), 
\end{equation}
and 
\begin{equation}\label{PropertiesfAndg}
f(u)\in L^\infty((0,\TM);L^\frac{1}{k}(\Omega)) \quad \textrm{and}\quad g(u)\in L^\infty((0,\TM);L^\frac{1}{l}(\Omega)), \textrm{ whenever }k,l\in (0,1].
\end{equation}
Further, globality and boundedness of $(u,v,w)$ (in  the sense of \eqref{ClassicalAndGlobability}) are ensured whenever (\textit{boundedness criterion}, below)  $u\in L^\infty((0,T_{max});L^p(\Omega))$, with some $p>1$ large, and uniformly with respect $t\in (0,T_{max})$: formally, 
\begin{equation}\label{BoundednessCriterio}
	\begin{array}{c}
		\textrm{If } \exists  \; L>0, p>\frac{n}{2}\; \Big| \; \displaystyle \int_\Omega u^p \leq L  \textrm{ on }  (0,\TM)\Rightarrow (u,v,w) \in (L^\infty((0, \infty);L^{\infty}(\Omega)))^3.
	\end{array}
\end{equation}
Since the arguments concerning local existence and boundedness criterion are standard, we do not justify them nor dedicate any lemma; details are achievable in \cite{TaoWanM3ASAttrRep} and \cite[Appendix A.]{TaoWinkParaPara}. 

Conversely, the mass conservation property \eqref{massConservation} comes from an integration over $\Omega$ of the first equation in problem \eqref{problem}. On the other hand, from this bound  we have \eqref{PropertiesfAndg} by noticing that
\begin{equation}\label{PropertiesfAndgConCostanti}
\int_\Omega f(u)^\frac{1}{k}\leq \alpha^\frac{1}{k}\int_\Omega u=\const{a} \quad \textrm{ and similarly } \int_\Omega g(u)^\frac{1}{l}\leq \const{za} \;  \textrm{ for all }   t\in (0,\TM).
\end{equation}
Inclusions \eqref{PropertiesfAndg}, as well as relations \eqref{PropertiesfAndgConCostanti}, will be of crucial importance when we analyze properties of solutions to the fully parabolic model \eqref{problem}. In particular, parabolic regularity theory allows to derive this result, which provides some uniform bounds for $\|v(\cdot,t)\|_{W^{1,r}(\Omega)}$ and $\|w(\cdot,t)\|_{W^{1,s}(\Omega)}$, with $r, s\geq 1$.
\begin{lemma}\label{LocalV}   
We have that $v$ and $w$ are such that
\begin{equation*}\label{Cg}
\int_\Omega |\nabla v(\cdot, t)|^r\leq \const{b} \quad \textrm{on } \,  (0,\TM)
\begin{cases}
\; \textrm{for all } r \in [1,\infty) & \textrm{if } k \in \left(0, \frac{1}{n}\right],\\
\;  \textrm{for all } r \in \left[1, \frac{n}{nk-1}\right) & \textrm{if } k \in \left(\frac{1}{n},1\right],
\end{cases}  
\end{equation*}
and
\begin{equation*}\label{Cg1}
\int_\Omega |\nabla w(\cdot, t)|^s\leq \const{c} \quad \textrm{on } \,  (0,\TM)
\begin{cases}
\; \textrm{for all } s \in [1,\infty) & \textrm{if } l \in \left(0, \frac{1}{n}\right],\\
\;  \textrm{for all } s \in \left[1, \frac{n}{nl-1}\right) & \textrm{if } l \in \left(\frac{1}{n},1\right].
\end{cases}  
\end{equation*}
\begin{proof}
The proof takes into account inclusions \eqref{PropertiesfAndg} and, depending on $\tau$, it is consequence of elliptic or parabolic regularity results; see, for instance, \cite{BrezisBook} and \cite{LSUBookInequality,HorstWink}.
\end{proof}
\end{lemma}
\section{Some preparatory tools}\label{PreliminariesSection}
In this section we collect some inequalities and further necessary results. 
\begin{lemma}\label{BoundsInequalityLemmaTecnicoYoung}  
Let $A,B \geq 0$, $d_1, d_2>0$ and $p>1$. Then for some $d, d_3>0$ we have 
\begin{equation}\label{InequalityForFinallConclusion}
A^{d_1}+B^{d_2}\geq 2^{-d}(A+B)^{d}-d_3, 
\end{equation}
and
\begin{equation}\label{InequalityA+BToPowerP}
(A+B)^p \leq 2^{p-1}(A^p+B^p).
\end{equation}
Moreover, let $d_4,d_5 >0$ be such that $d_4+d_5<1$. Then for all  
$\epsilon>0$ 
\begin{equation}\label{LemmaEsponenti}
A^{d_4}B^{d_5} \leq \epsilon(A+B)+\const{d}.
\end{equation}
\begin{proof}
The proofs are available, respectively, in \cite[Lemma 3.3]{MarrasViglialoroMathNach}, \cite[Theorem 1]{Jameson_2014Inequality} and \cite[Lemma 4.3]{frassuviglialoro1}.
\end{proof}
\end{lemma}
\begin{lemma}\label{EllipticEhrlingSystemLemma} 
Let $\Omega\subset \R^n$, $n\geq 1$, be a bounded and smooth domain and $\delta>0$. Then for any nonnegative $g\in C^1(\bar{\Omega})$, the solution $0\leq \psi\in C^{2,\kappa}(\bar{\Omega})$, $0<\kappa<1$, of the problem
\begin{equation*}
\begin{cases}
0=\Delta \psi+g-\delta \psi & \textrm{in } \Omega,\\
\psi_{\nu}=0 & \textrm{on } \partial \Omega,
\end{cases}
\end{equation*}
has the following property: For any $\hat{c},\sigma>0$ and  $\overline{p}\in(1,\infty)$, there exists $\tilde{c}=\tilde{c}(\sigma,\overline{p}) >0$ such that 
\begin{equation}\label{EhrlingTypeInequalityWithMass} 
\hat{c}\int_\Omega \psi^{\overline{p}+1} \leq \sigma  \int_\Omega g^{\overline{p}+1} +\frac{\tilde{c}}{|\Omega|^{\overline{p}}}\Big( \int_\Omega g\Big)^{\overline{p}+1}.\end{equation}
\end{lemma}
\begin{proof}
A detailed proof of \eqref{EhrlingTypeInequalityWithMass} can be found in \cite[Lemma 3.1]{ViglialoroMatNacAttr-Repul}. (See also \cite[Lemma 2.2]{WinklerHowFar}.)
\end{proof}
\section{A priori estimates and proof of the  theorems}\label{EstimatesAndProofSection}
Let $(u,v,w)$ be the local classical solution to models \eqref{problem} and \eqref{problemConFunzioniTempo}. Whenever we show that $u\in L^\infty((0,\TM);L^p(\Omega))$, for some $p>\frac{n}{2}$,  we can take advantage from the boundedness criterion given in \eqref{BoundednessCriterio} and directly obtain that, indeed,  $u\in L^\infty((0,\infty);L^\infty(\Omega))$; at this point, regularity theories applied to the equations of $v$ and $w$ entail that also $v,w$ belong to $L^\infty((0,\infty);L^\infty(\Omega)).$

So our goal in this section is showing the inclusion $u\in L^\infty((0,\TM);L^p(\Omega))$, for some $p>\frac{n}{2}$; we will prove this by establishing an absorptive differential inequality for the functional $\displaystyle y(t):=\int_\Omega u^p$. To be precise, we aim at deriving this inequality for $y$ 
\begin{equation}\label{EstimFun}
y'(t)+ \const{ac} \int_\Omega \lvert \nabla u^\frac{p}{2}\rvert^2 \leq  \const{ad}
\quad \textrm{on } (0, T_{max}).
\end{equation}
In fact, we can successively exploit the Gagliardo--Nirenberg inequality (see \cite{Nirenber_GagNir_Ineque}) combined with \eqref{InequalityA+BToPowerP}
(used in the sequel without mentioning), so to write 
\begin{equation*} 
\int_\Omega u^{p}=\lvert \lvert u^\frac{p}{2}\lvert \lvert_{L^2(\Omega)}^2 
\leq \const{ae}  \lvert \lvert\nabla u^\frac{p}{2}\lvert \lvert_{L^2(\Omega)}^{2 \theta} \lvert \lvert u^\frac{p}{2}\lvert \lvert_{L^\frac{2}{p}(\Omega)}^{2(1-\theta)} + \const{ae} \lvert \lvert u^\frac{p}{2}\lvert \lvert^2_{L^\frac{2}{p}(\Omega)} \quad \textrm{ for all } t \in(0,T_{max}),
 \end{equation*}
 with
\[0<\theta=\frac{\frac{np}{2}(1-\frac{1}{p})}{1-\frac{n}{2}+\frac{np}{2}}<1.\]
In turn, by taking into consideration bound \eqref{massConservation}, the above inequality leads to 
\begin{equation}\label{GN3}
\int_\Omega u^{p}\leq \const{ag} \Big(\int_\Omega \lvert \nabla u^\frac{p}{2}\rvert^2\Big)^{\theta}+ \const{ag}\quad \textrm{on } \, (0,T_{max}),
\end{equation}
so that by manipulating it and inserting the result into \eqref{EstimFun}, we can observe also by virtue of \eqref{InequalityForFinallConclusion} that $y$ satisfies this initial problem
\begin{equation*}\label{MainInitialProblemWithM}
\begin{cases}
y'(t)\leq \const{ai} - \const{aj} y^{\frac{1}{\theta}}(t)\quad \textrm{for all } t \in (0,T_{max}),\\
y(0)=\int_\Omega u_0^p.
\end{cases}
\end{equation*}
Consequently, an ODE comparison principle implies that 
$\displaystyle \int_\Omega u^p\leq \max\left\{y(0),\left(\frac{\const{ai}}{\const{aj}}\right)^\theta\right\}:=L$ for all $t\in(0,T_{max})$, providing the desired inclusion.
\subsection*{Analysis of solutions to model \eqref{problem} with $\tau=0$; proof of Theorem \ref{MainTheorem}}
\begin{lemma}\label{Estim_general_For_u^pLemmaEllittico} 
Let either $0<k<l$ or $k,l\in\left(0,\frac{2}{n}\right)$. Then for any $p > \max\left\{l,\frac{l(nl-2)}{n},\frac{n}{2}\right\}$, $u$ fulfills for some $L>0$  
\begin{equation}\label{Up}
\int_\Omega u^p \leq L \quad \textrm{for all}\quad t \in (0,T_{max}).
\end{equation}
Moreover, we have the same conclusion if $k=l$ and $\chi \alpha - \xi \gamma_0<0$, or $l=k\in\left(0,\frac{2}{n}\right)$ and $\chi \alpha - \xi \gamma_0\geq 0.$ 
\begin{proof}
By testing the first equation of problem \eqref{problem} with $u^{p-1}$, using its boundary conditions and taking into account the second and the third equation with $\tau=0$,  we have
\begin{equation*}
\begin{split}
\frac{d}{dt} \int_\Omega u^p &=p \int_\Omega u^{p-1}u_t
= -p (p-1) \int_\Omega u^{p-2} |\nabla u|^2 - (p-1) \chi \int_\Omega u^p (v-f(u)) + (p-1)\xi \int_\Omega u^p (w-g(u))\\
&\leq -p(p-1) \int_\Omega u^{p-2} |\nabla u|^2 + (p-1) \chi \int_\Omega u^p f(u) + (p-1)\xi \int_\Omega u^p w - (p-1)\xi \int_\Omega u^p g(u) \quad \textrm{on }\, (0,\TM).
\end{split}  
\end{equation*}  
Now, we recall the definitions of $f$ and $g$ given in \eqref{f}, exploit $u < (u+1)$ and inequality \eqref{InequalityA+BToPowerP} so to obtain
\begin{equation}\label{Est1}
\begin{split}
\frac{d}{dt} \int_\Omega u^p &\leq -p(p-1) \int_\Omega u^{p-2} |\nabla u|^2 + (p-1)\chi \alpha \int_\Omega u^{p+k} + (p-1)\xi \int_\Omega u^p w 
- (p-1)\xi \gamma_0 \int_\Omega u^p (u+1)^l\\
& \leq -p(p-1) \int_\Omega u^{p-2} |\nabla u|^2 + (p-1)\chi \alpha \int_\Omega u^{p+k} + (p-1)\xi \int_\Omega u^p w - (p-1)\xi \gamma_0 \int_\Omega u^{p+l} \quad \textrm{for all }\, t \in (0,\TM).
\end{split}
\end{equation}
As to the third integral on the right-hand side of \eqref{Est1}, an application of Young's inequality, Lemma \ref{EllipticEhrlingSystemLemma} with $\psi=w$ and 
$\bar{p}=\frac{p}{l}>1$ and bound \eqref{InequalityA+BToPowerP} provide for 
$\epsilon_1, \sigma, \tilde{\sigma}>0$ 
\begin{equation}\label{Est2}
\begin{split}
(p-1)\xi \int_\Omega u^p w &\leq \epsilon_1 \int_{\Omega} u^{p+l} + \const{e} \int_{\Omega} w^{\bar{p}+1}
\leq \epsilon_1 \int_{\Omega} u^{p+l} + \sigma \int_{\Omega} (g(u))^{\overline{p}+1} 
+ \const{f} \left(\int_{\Omega} g(u)\right)^{\overline{p}+1}\\
&\leq \epsilon_1 \int_{\Omega} u^{p+l} + \sigma \int_{\Omega} (\gamma_1 (u+1)^l)^{\overline{p}+1} + \const{f} \left(\int_{\Omega} \gamma_1 (u+1)^l\right)^{\overline{p}+1}\\
&\leq \epsilon_1 \int_{\Omega} u^{p+l} + \tilde{\sigma} \int_{\Omega} u^{l(\bar{p}+1)} + \const{i} \left(\int_{\Omega} u^l\right)^{\overline{p}+1}+ \const{l}
\quad \textrm{on } (0,\TM).
\end{split}
\end{equation}
If $l\in(0,1]$ the last integral in \eqref{Est2} is controlled by a constant thanks Hölder's inequality and the mass conservation property \eqref{massConservation}. On the other hand, if $l>1$ we can estimate the term $\displaystyle \left(\int_{\Omega} u^l\right)^{\bar{p}+1}$ by applying the Gagliardo--Nirenberg inequality, in this way: 
\begin{equation*}
\const{i} \left(\int_{\Omega} u^l \right)^{\frac{p+l}{l}}= \const{i}
 \|u^{\frac{p}{2}}\|_{L^{\frac{2l}{p}}(\Omega)}^{\frac{2(p+l)}{p}} 
\leq \const{o} \|\nabla u^{\frac{p}{2}}\|_{L^2(\Omega)}^{\frac{2(p+l)}{p}\theta_1} \|u^{\frac{p}{2}}\|_{L^{\frac{2}{p}}(\Omega)}^{\frac{2(p+l)}{p}(1-\theta_1)} 
+ \const{o} \|u^{\frac{p}{2}}\|_{L^{\frac{2}{p}}(\Omega)}^{\frac{2(p+l)}{p}} \quad \textrm{ for all } t \in(0,T_{max}),
\end{equation*}
where for $p$ as in our assumptions we have
\[
0<\theta_1=\frac{1-\frac{1}{l}}{1+\frac{2}{np}-\frac{1}{p}}<1 \quad \textrm{ and } \quad 0<\frac{(p+l)}{p}\theta_1 <1.
\]
Hence, by recalling the mass conservation property \eqref{massConservation}, the Young and above inequalities entail for any $\epsilon_2>0$ 
\begin{equation}\label{GN11}
\const{i} \left(\int_\Omega u^l \right)^{\frac{p+l}{l}} \leq \const{r} \Big(\int_\Omega |\nabla u^\frac{p}{2}|^2\Big)^{\frac{(p+l)}{p}\theta_1}+ \const{r}
\leq \epsilon_2 \int_\Omega |\nabla u^\frac{p}{2}|^2  + \const{t} \quad \textrm{on } (0,T_{max}).
\end{equation}
As a consequence, in both case, by putting estimates \eqref{Est2} and \eqref{GN11} into \eqref{Est1} we get for  all $t \in (0,\TM)$
\begin{equation}\label{Est3}
\begin{split}
\frac{d}{dt} \int_\Omega u^p & \leq -p(p-1) \int_\Omega u^{p-2} |\nabla u|^2 + (p-1)\chi \alpha \int_\Omega u^{p+k} 
+(\epsilon_1+\tilde{\sigma}- (p-1)\xi \gamma_0) \int_\Omega u^{p+l} 
+\epsilon_2 \int_\Omega |\nabla u^\frac{p}{2}|^2+ \const{u}.
\end{split}
\end{equation}
Now we analyze the case $0<k<l$ and by applying the Young inequality to the second integral at the right-hand side of the above estimate yields for $\epsilon_3>0$ 
\begin{equation*}
(p-1)\chi\alpha \int_{\Omega} u^{p+k} \leq \epsilon_3 \int_{\Omega} u^{p+l} + \const{aa}
\quad \textrm{for all  } t \in (0,T_{max}),
\end{equation*}
so that, by plugging the latter into \eqref{Est3}, and also in view of the identity
\[
\int_\Omega u^{p-2} |\nabla u|^2 = \frac{4}{p^2} \int_\Omega |\nabla u^{\frac{p}{2}}|^2 \quad \textrm{on  } (0,T_{max}),
\] 
we have
\begin{equation*}\label{EstT}
\begin{split}
\frac{d}{dt} \int_\Omega u^p & \leq \left(\epsilon_2-\frac{4(p-1)}{p}\right) \int_\Omega |\nabla u^{\frac{p}{2}}|^2 
+(\epsilon_1+\tilde{\sigma}+\epsilon_3 - (p-1)\xi \gamma_0) \int_\Omega u^{p+l} + \const{ab} \quad \textrm{on }\, (0,\TM).
\end{split}
\end{equation*}
Finally, by choosing $\epsilon_1, \epsilon_2,\epsilon_3, \tilde{\sigma}$ such that $\epsilon_2  < \frac{4(p-1)}{p}$ and $\epsilon_1+\tilde{\sigma}+\epsilon_3\leq(p-1)\xi\gamma_0$, we arrive at similar inequality as in \eqref{EstimFun}. (Note that in the case $l\in(0,1]$ one can directly take $\epsilon_2=0$.)

Let turn our attention to the case $l,k\in \left(0,\frac{2}{n}\right)$. Thanks to the Gagliardo--Nirenberg inequality and to the mass conservation property \eqref{massConservation}, the estimate of the second integral of the right-hand side of \eqref{Est3} becomes 
\begin{equation*}
\begin{split}
(p-1)\chi \alpha \int_\Omega u^{p+k}&= (p-1)\chi \alpha
\|u^{\frac{p}{2}}\|_{L^{\frac{2(p+k)}{p}}(\Omega)}^{\frac{2(p+k)}{p}}
\leq \const{v} \|\nabla u^{\frac{p}{2}}\|_{L^2(\Omega)}^{\frac{2(p+k)}{p}\theta_2} 
\|u^{\frac{p}{2}}\|_{L^{\frac{2}{p}}(\Omega)}^{\frac{2(p+k)}{p}(1-\theta_2)} 
+ \const{v} \|u^{\frac{p}{2}}\|_{L^{\frac{2}{p}}(\Omega)}^{\frac{2(p+k)}{p}}\\
&\leq \const{j} \left(\int_\Omega |\nabla u^{\frac{p}{2}}|^2\right)^{\frac{(p+k)}{p}\theta_2} + \const{j} 
\quad \textrm{on } (0,\TM),
\end{split}
\end{equation*}
where
\[
\theta_2=\frac{\frac{p}{2} - \frac{p}{2(p+k)}}{\frac{p}{2}+\frac{1}{n}-\frac{1}{2}} \in (0,1).
\]
Since $\frac{(p+k)}{p}\theta_2<1$ for $0<k<\frac{2}{n}$, an application of the Young inequality gives for $\epsilon_4>0$ 
\begin{equation}\label{GN2}
(p-1)\chi \alpha \int_\Omega u^{p+k} \leq \epsilon_4 \int_\Omega |\nabla u^{\frac{p}{2}}|^2 + \const{y} \quad \textrm{for all } t \in (0,\TM).
\end{equation}
By plugging this gained estimate into bound \eqref{Est3},  we reason as above and the arbitrariness of the constants associated to the terms $\int_\Omega |\nabla u^\frac{p}{2}|^2$ and $\int_\Omega u^{p+l}$ allows us to establish for $y(t)$ an estimate as in \eqref{EstimFun}. 

Finally, if we take $k=l$ in \eqref{Est3}, we arrive at the following estimate: 
\begin{equation*}
\frac{d}{dt} \int_\Omega u^p 
\leq -p(p-1) \int_\Omega u^{p-2} |\nabla u|^2 + ((p-1)(\chi \alpha-\xi \gamma_0) + \epsilon_1 +\tilde{\sigma}) \int_\Omega u^{p+l} 
+ \epsilon_2 \int_\Omega \lvert \nabla u^\frac{p}{2}\rvert^2 + \const{u} \quad \textrm{for all }\, t \in (0,\TM).
\end{equation*}
If $\chi \alpha - \xi \gamma_0<0$, for proper constant $\epsilon_1, \tilde{\sigma}$, we can make $(p-1)(\chi \alpha-\xi \gamma_0) + \epsilon_1 +\tilde{\sigma}\leq 0$ and consequently neglect the term associated to
$\displaystyle \int_\Omega u^{p+l}$;  on the other hand if $\chi \alpha - \xi \gamma_0\geq 0$ and $l\in\left(0,\frac{2}{n}\right)$ we may treat such integral similarly to \eqref{GN2}. In both cases we conclude.
\end{proof}
\end{lemma}
\subsection*{Analysis of solutions to model \eqref{problemConFunzioniTempo}; proof of Theorem \ref{MainTheorem3}}
\begin{lemma}\label{Estim_general_For_u^pLemmaEllitticoTempo} 
Let $l>1$ and $k<l$. Then there exists $q>l$ such that whenever 
$p>\max\{\frac{q}{n}(qn-2),\frac{n}{2}\}$ we have that $u$ obeys the same relation in \eqref{Up}, possibly with a different constant $L$.

Moreover, the same conclusion holds true whenever either $l\in(0,1]$ and $k<l$, or $k=l$ and $\chi \alpha - \xi \gamma_0<0$, or $l=k\in\left(0,\frac{2}{n}\right)$ and 
$\chi \alpha - \xi \gamma_0\geq 0$.
\begin{proof}
By reasoning as in the proof of Lemma \ref{Estim_general_For_u^pLemmaEllittico}, we deduce the following estimate on $(0,\TM)$:
\begin{equation}\label{EstEst1}
\begin{split}
\frac{d}{dt} \int_\Omega u^p &= -p (p-1) \int_\Omega u^{p-2} |\nabla u|^2 
- (p-1) \chi \int_\Omega u^p \left(\frac{1}{|\Omega|}\int_\Omega f(u)-f(u)\right)
+ (p-1)\xi \int_\Omega u^p \left(\frac{1}{|\Omega|}\int_\Omega g(u)-g(u)\right)\\
&\leq -p(p-1) \int_\Omega u^{p-2} |\nabla u|^2 + (p-1) \chi \alpha \int_\Omega u^{p+k} 
+ \frac{(p-1)\xi}{|\Omega|}\int_\Omega g(u) \int_\Omega u^p
- (p-1)\xi \gamma_0 \int_\Omega u^{p+l}\\
&\leq -p(p-1) \int_\Omega u^{p-2} |\nabla u|^2 + (p-1) \chi \alpha \int_\Omega u^{p+k}
+ \const{am} \int_\Omega u^p + \const{an} \int_\Omega u^l \int_\Omega u^p
- (p-1)\xi \gamma_0 \int_\Omega u^{p+l}.
\end{split}  
\end{equation}
Since $k<l$, by exploiting Young's inequality we obtain for $\epsilon_5>0$ 
\begin{equation*}
(p-1) \chi \alpha \int_\Omega u^{p+k} \leq \epsilon_5 \int_\Omega u^{p+l} +\const{ap} 
\quad \textrm{on } (0,\TM).
\end{equation*}
If $0<l\leq1$, the product of the two integrals $\int_\Omega u^l\int_\Omega u^p$ in \eqref{EstEst1} can be treated by exploiting the mass conservation property \eqref{massConservation} and bound \eqref{GN3}, so to get for $\epsilon_6>0$ 
\[
\int_\Omega u^l \int_\Omega u^p \leq \const{aq} \int_\Omega u^p \leq \epsilon_6 
\int_\Omega \lvert \nabla u^\frac{p}{2}\rvert^2 + \const{ar} \quad \textrm{for all } t \in (0,\TM).
\]
On the other hand, if $l>1$ a twice application of the H\"{o}lder inequality gives 
\begin{equation*}  
\int_\Omega u^l \int_\Omega u^p \leq \const{as} \left[\left(\int_\Omega u^q \right)^{\frac{p}{q}+1}\right]^{\frac{l}{p+q}} \left(\int_\Omega u^{p+l}\right)^{\frac{p}{p+l}}
\quad \textrm{on } (0,\TM).
\end{equation*}
The restriction $q>l$ implies $\frac{l}{p+q}+\frac{p}{p+l}<1$, so that by invoking inequality \eqref{LemmaEsponenti} we can get for $\epsilon_7>0$
\begin{equation*}  
\int_\Omega u^l \int_\Omega u^p \leq \epsilon_7 \int_\Omega u^{p+l} 
+ \epsilon_7 \left(\int_\Omega u^q \right)^{\frac{p+q}{q}} + \const{at}
\quad \textrm{on } (0,\TM).
\end{equation*}
At this point, we can manipulate the term $\displaystyle \left(\int_\Omega u^q \right)^{\frac{p+q}{q}}$ as in \eqref{GN11}. Rephrasing estimate \eqref{EstEst1} in terms of what derived, the same motivations used in Lemma \ref{Estim_general_For_u^pLemmaEllittico} give the claim.
\end{proof}
\end{lemma}
\subsection*{Analysis of solutions to model \eqref{problem} with $\tau=1$; proof of Theorem \ref{Main1Theorem}}
\begin{lemma}\label{Estim_general_For_u^pLemmaParabolico1} 
For any $p>1$ and $r, s>2$, we have that $u$ is such that for $\epsilon_8, \epsilon_9>0$ 
\begin{equation}\label{ClaimParab}
\frac{d}{dt} \int_\Omega u^p \leq (\epsilon_8 + \epsilon_9 - p(p-1)) \int_\Omega u^{p-2} |\nabla u|^2 + \int_\Omega |\nabla v|^r + \int_\Omega |\nabla w|^s + c_{34} \int_\Omega u^{\frac{pr}{r-2}} + c_{36} \int_\Omega u^{\frac{ps}{s-2}}
\quad \textrm{for all } t \in (0,\TM).
\end{equation}
\begin{proof}
Testing procedures imply the following estimate
\begin{equation}\label{Parab1} 
\frac{d}{dt} \int_\Omega u^p = -p (p-1) \int_\Omega u^{p-2} |\nabla u|^2 
+ p(p-1) \chi \int_\Omega u^{p-1} \nabla u \cdot \nabla v 
- p(p-1)\xi \int_\Omega u^{p-1} \nabla u \cdot \nabla w  \quad \textrm{on } (0,\TM),
\end{equation}
whereas a twice application of Young's inequality infers for 
$\epsilon_8>0$ and for all $t \in (0,\TM)$
\begin{equation}\label{ParabYoung1}
p(p-1) \chi \int_\Omega u^{p-1} \nabla u \cdot \nabla v \leq 
\epsilon_8 \int_\Omega u^{p-2} |\nabla u|^2 + \const{az}\int_\Omega u^p |\nabla v|^2
\leq \epsilon_8 \int_\Omega u^{p-2} |\nabla u|^2 
+ \int_\Omega |\nabla v|^r + \const{ba} \int_\Omega u^{\frac{pr}{r-2}},
\end{equation} 
and for $\epsilon_9>0$ on $(0,\TM)$
\begin{equation}\label{ParabYoung2}
-p(p-1) \xi \int_\Omega u^{p-1} \nabla u \cdot \nabla w \leq 
\epsilon_9 \int_\Omega u^{p-2} |\nabla u|^2 + \const{bb} \int_\Omega u^p |\nabla w|^2
\leq \epsilon_9 \int_\Omega u^{p-2} |\nabla u|^2 
+ \int_\Omega |\nabla w|^s + \const{bc} \int_\Omega u^{\frac{ps}{s-2}}.
\end{equation}
The claim is deduced by simply plugging relations \eqref{ParabYoung1} and \eqref{ParabYoung2} into bound \eqref{Parab1}.
\end{proof}
\end{lemma}
\begin{lemma}\label{Estim_general_For_u^pLemmaParabolico2} 
Let either $l,k \in \left(0,\frac{1}{n}\right]=I$, or $l,k \in \left(\frac{1}{n},\frac{1}{n}+\frac{2}{n^2+4}\right)=J$, or $l\in I$ and $k\in J$ (or viceversa). Then for some $p>\frac{n}{2}$ and some $L_1>0$, we have that $u$ complies with
\begin{equation*}\label{Up1}
\int_\Omega u^p \leq L_1 \quad \textrm{for all}\quad t \in (0,T_{max}).
\end{equation*}
\begin{proof}
Let us set $p_0=\frac{n}{2}$, $\rho_0=\frac{2}{n}$ and $\phi(p,\rho,\sigma):=\frac{p\sigma}{\sigma-2}-(p+\rho)$, which is defined for all $p>1$, $\rho>0$ and 
$\sigma>2$. 
Since $\phi(p_0, \rho_0, \sigma)=\frac{n^2-2 \sigma+4}{n (\sigma-2)}<0$ for all $\sigma>\frac{n^2+4}{2}$, by continuity there are
$p>p_0$ and $\rho<\rho_0$ such that $\phi(p,\rho,\sigma)<0$, i.e.
\begin{equation}\label{SigmaRho}
\frac{p\sigma}{\sigma-2} < p+\rho.
\end{equation} 
Being our scope controlling the terms $\int_\Omega |\nabla v|^r$, 
$\int_\Omega u^{\frac{pr}{r-2}}$ and  $\int_\Omega |\nabla w|^s$, 
$\int_\Omega u^{\frac{ps}{s-2}}$ appearing in \eqref{ClaimParab}, let us take into consideration the regularity of $v$ and $w$, and in particular the functional spaces where $\nabla v$ and $\nabla w$ belong, and let us study three different cases. 
\begin{itemize}
\item [$\bullet$] Case $k,l \in I$. From Lemma \ref{LocalV}, it is known that $\nabla v \in L^{\infty}((0,\TM);L^r(\Omega))$ and $\nabla w \in  L^{\infty}((0,\TM);L^s(\Omega))$ for all $r,s >1$, so that we may pick $r=s=\sigma>\frac{n^2+4}{2}>2$, and by applying \eqref{SigmaRho} we have for all $t \in (0,\TM)$
\begin{equation}\label{ParabYoung3}
p(p-1) \chi \int_\Omega u^{p-1} \nabla u \cdot \nabla v \leq \epsilon_8 \int_\Omega u^{p-2} |\nabla u|^2 +\int_\Omega |\nabla v|^r + \const{ba} \int_\Omega u^{\frac{pr}{r-2}}
\leq \epsilon_8 \int_\Omega u^{p-2} |\nabla u|^2 + \int_\Omega u^{p+\rho}+\const{be}, 
\end{equation} 
and on $(0,\TM)$
\begin{equation}\label{ParabYoung4}
-p(p-1) \xi \int_\Omega u^{p-1} \nabla u \cdot \nabla w 
\leq \epsilon_9 \int_\Omega u^{p-2} |\nabla u|^2 
+ \int_\Omega |\nabla w|^s + \const{bc} \int_\Omega u^{\frac{ps}{s-2}}
\leq \epsilon_9 \int_\Omega u^{p-2} |\nabla u|^2 + \int_\Omega u^{p+\rho} + \const{bg}.
\end{equation}
\item [$\bullet$] Case $k \in I$ and $l\in J$ (or, viceversa, $k\in J$ and $l \in I$).
We discuss the first situation, for which accordingly to Lemma \ref{LocalV}, we have $\nabla v \in L^{\infty}((0,\TM);L^r(\Omega))$, for all $r\geq 1$ and $\nabla w \in  L^{\infty}((0,\TM);L^s(\Omega))$ for all $1\leq s< \frac{n}{nl-1}$. (The second case is totally analogous.) Clearly estimate \eqref{ParabYoung3} still holds, while for the validity of bound \eqref{ParabYoung4} we have to check that 
$\frac{n}{nl-1} > \frac{n^2+4}{2}$ for all $l\in J$. Since $\Lambda(l):=\frac{n}{nl-1}$ is decreasing in $J$ and $\Lambda\left(\frac{1}{n}+\frac{2}{n^2+4}\right)=\frac{n^2+4}{2}$, we conclude. 
 \item [$\bullet$] Case $l, k\in J$.
By following for $k$ the same arguments used previously for $l$, we can ensure  that \eqref{ParabYoung3} and \eqref{ParabYoung4} are satisfied.
\end{itemize}
As a consequence of this, by putting bounds \eqref{ParabYoung3} and \eqref{ParabYoung4} into 
estimate \eqref{Parab1} and by manipulating the term $\displaystyle \int_\Omega u^{p+\rho}$, with $\rho<\frac{2}{n}$ as the integral in \eqref{GN2}, we are in the position to deduce the claim.
\end{proof}
\end{lemma}
\section{Numerical simulations}\label{SectionSimulations}
\begin{figure}[htbp]
\centering
\subfigure[Evolution in time of $\max u$.]{
    \includegraphics[width=8.7cm]{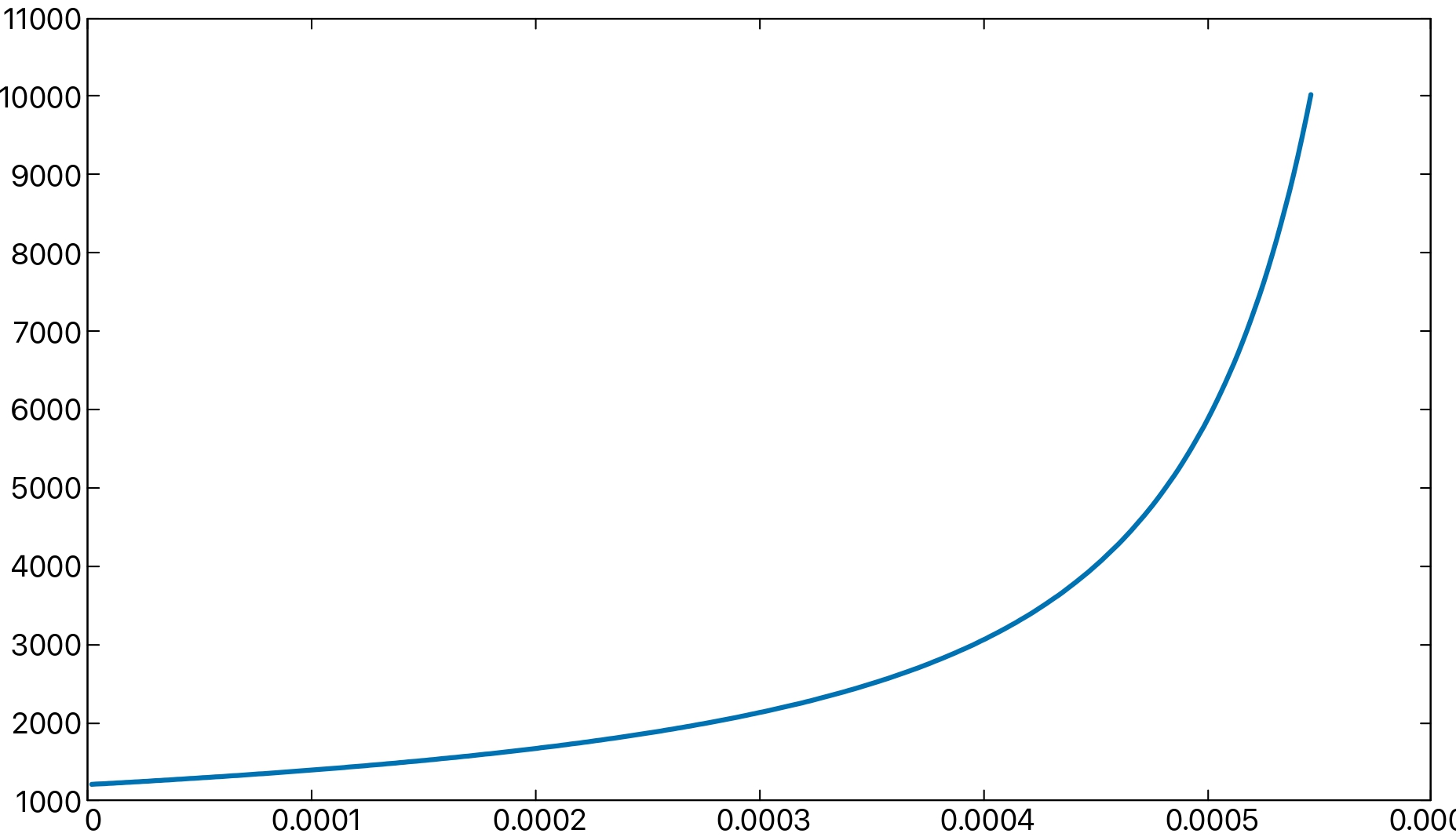}
    }
\hspace{0.2cm}
\subfigure[Evolution in time of $\max u$.]{
    \includegraphics[width=8.7cm]{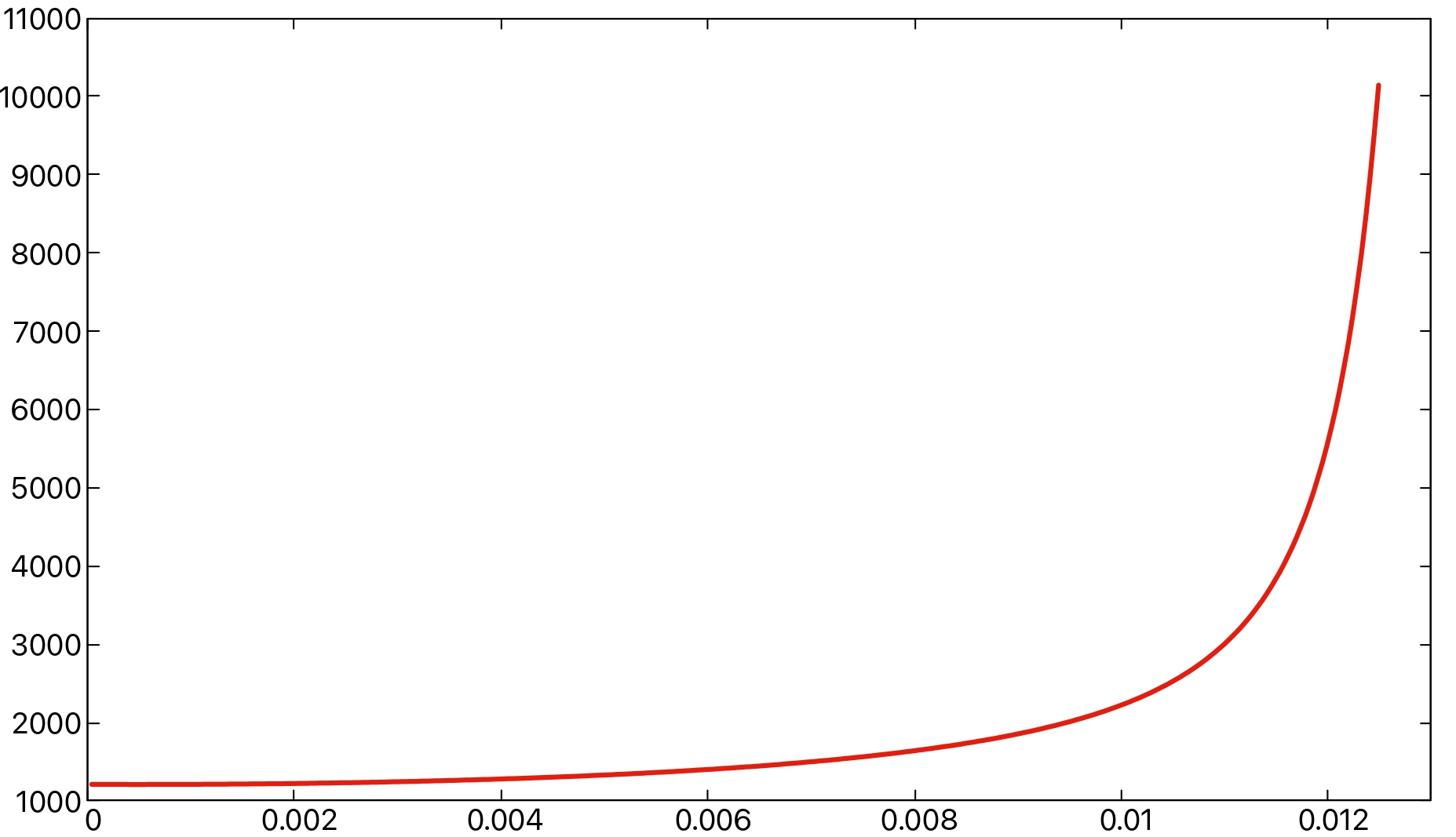}
}
\\
\subfigure[Comparison of $\max u$.]{
    \includegraphics[width=8.7cm]{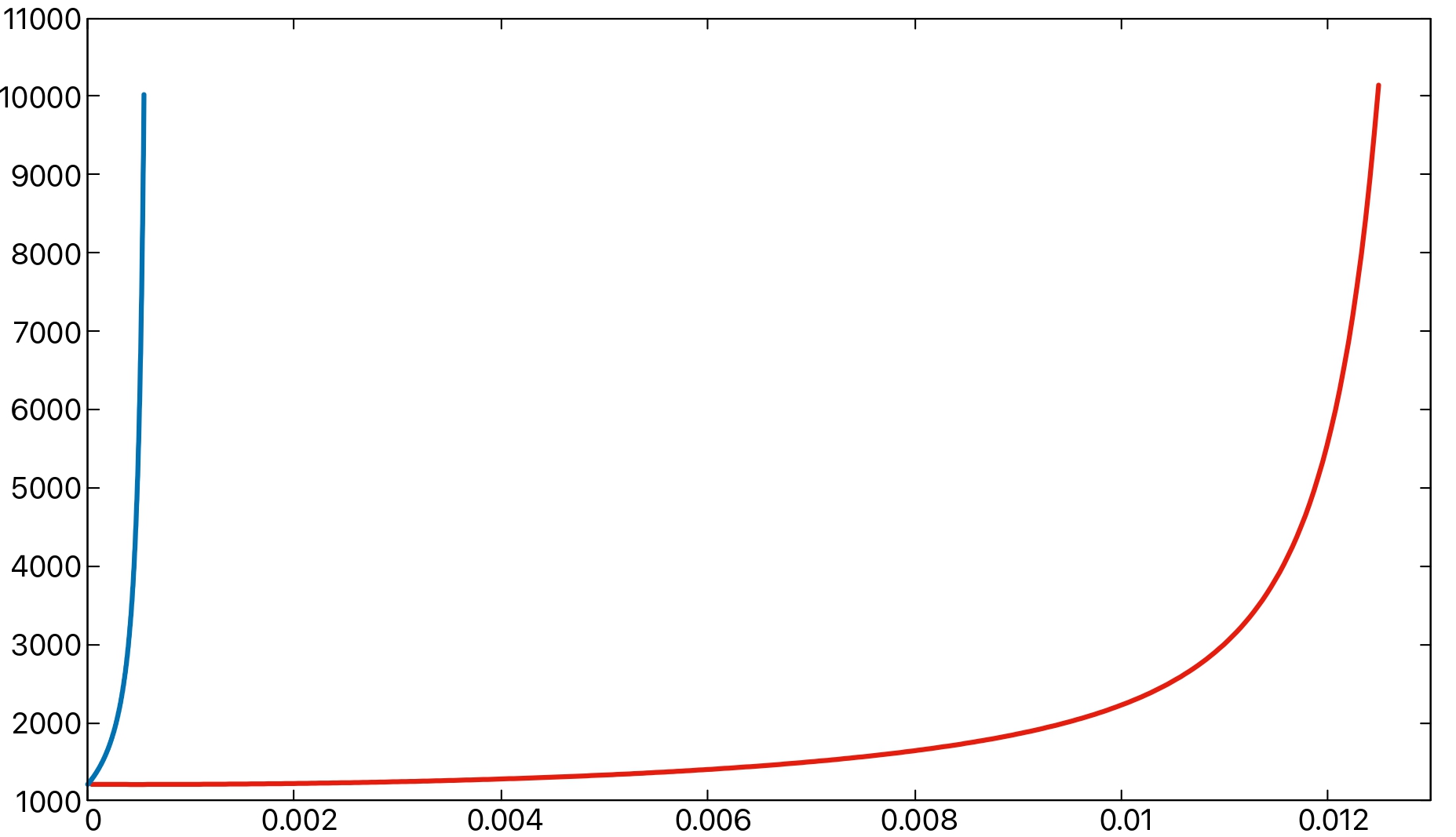}\label{SubFigCompar}}
 \hspace{0.2cm}
\subfigure[Representation of $\delta$-formations.]{
    \includegraphics[width=8.7cm]{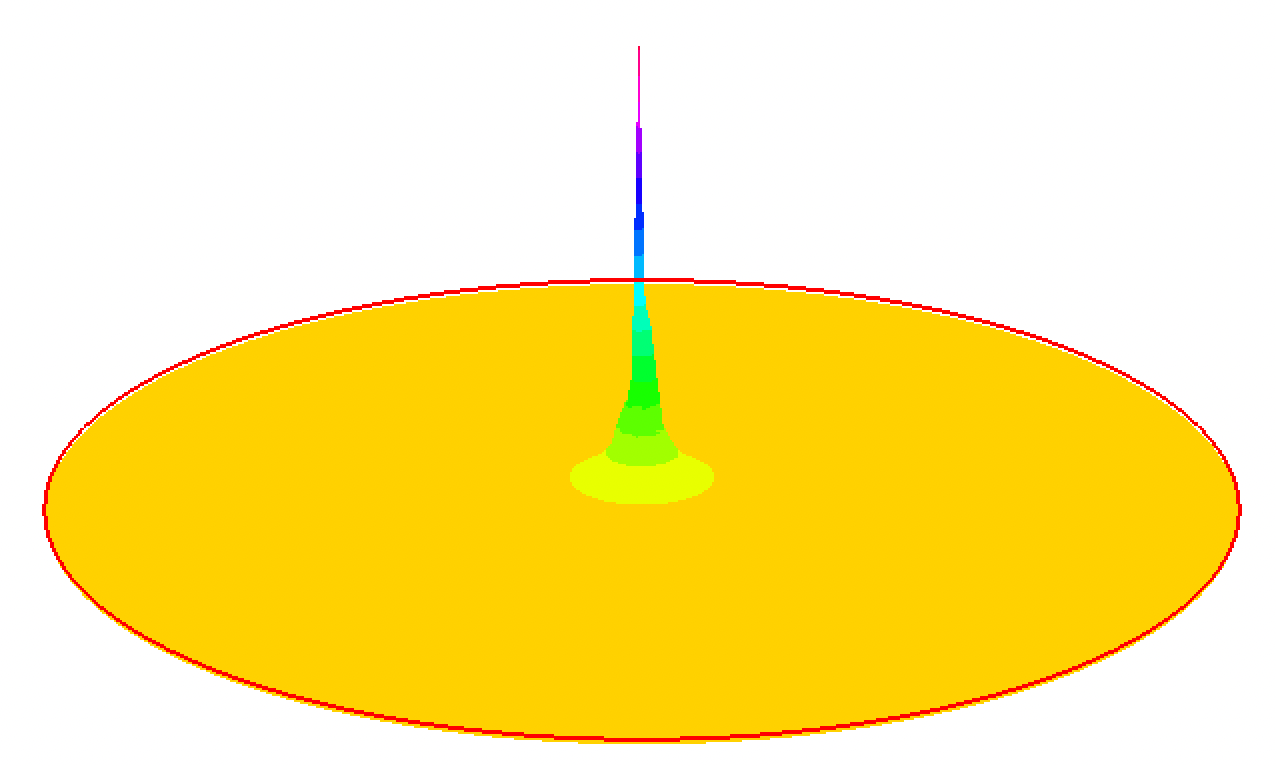}
}
\caption{{\bf{Blow-up case.}} Simulations of model \eqref{problem} for $\tau=0$ (blue color) and $\tau=1$ (red color). Evolution of the cells' density: $\TM\approx 0.00054$ (top left),  $\TM\approx 0.0012$ (top right), comparison of behaviors (bottom left), qualitative representation of finite time blow-up at a spatial point (bottom right). The quantitative values are indicated on the axes.}\label{fig:Blow-UPbothaCases} 
\end{figure}
\begin{figure}[htbp]
\centering
\subfigure[Evolution in time of $\max u$, $\max v$ and $\max w$.]{
    \includegraphics[width=8.7cm]{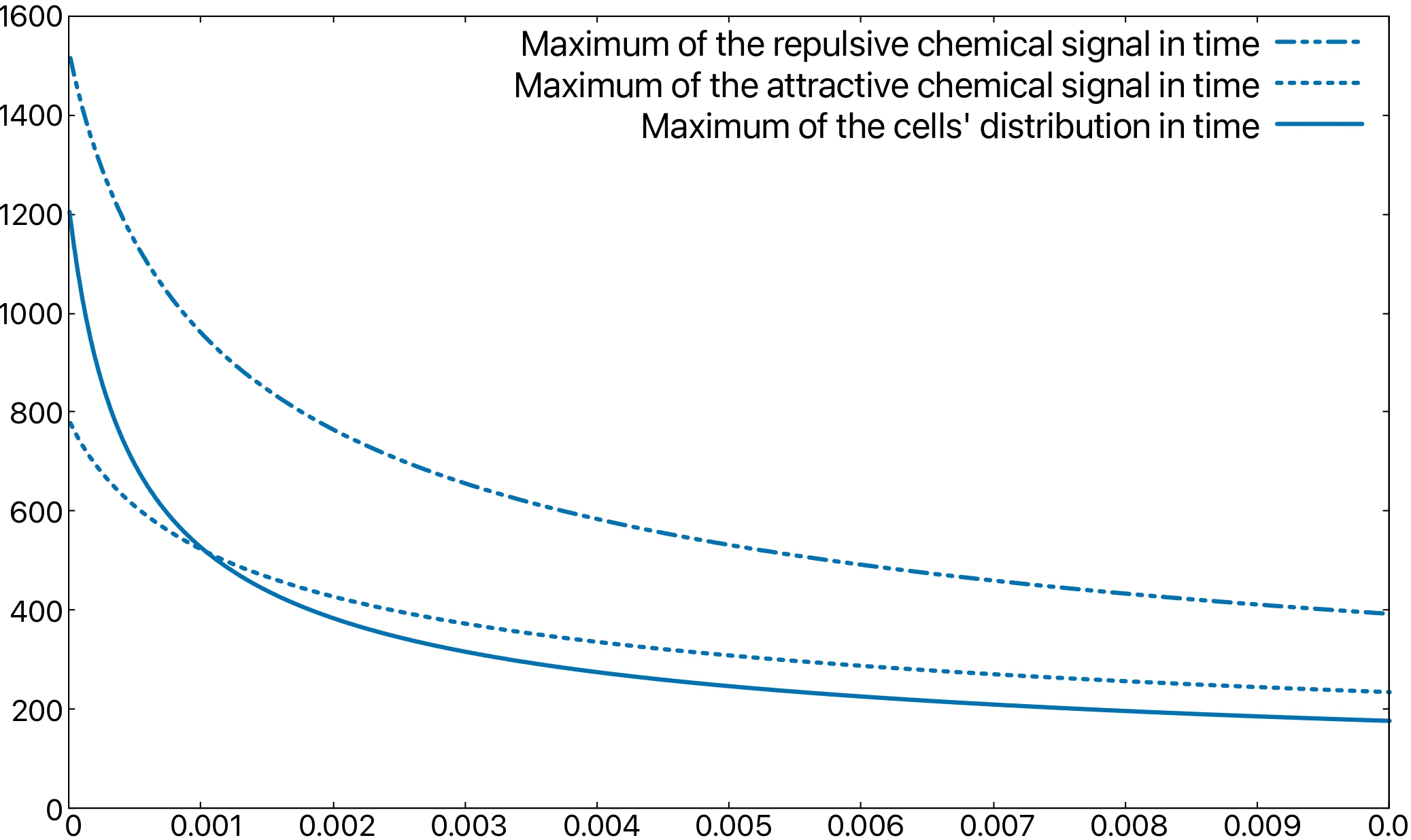}
    }
\hspace{0.2cm}
\subfigure[Evolution in time of $\max u$, $\max v$ and $\max w$.]{
    \includegraphics[width=8.7cm]{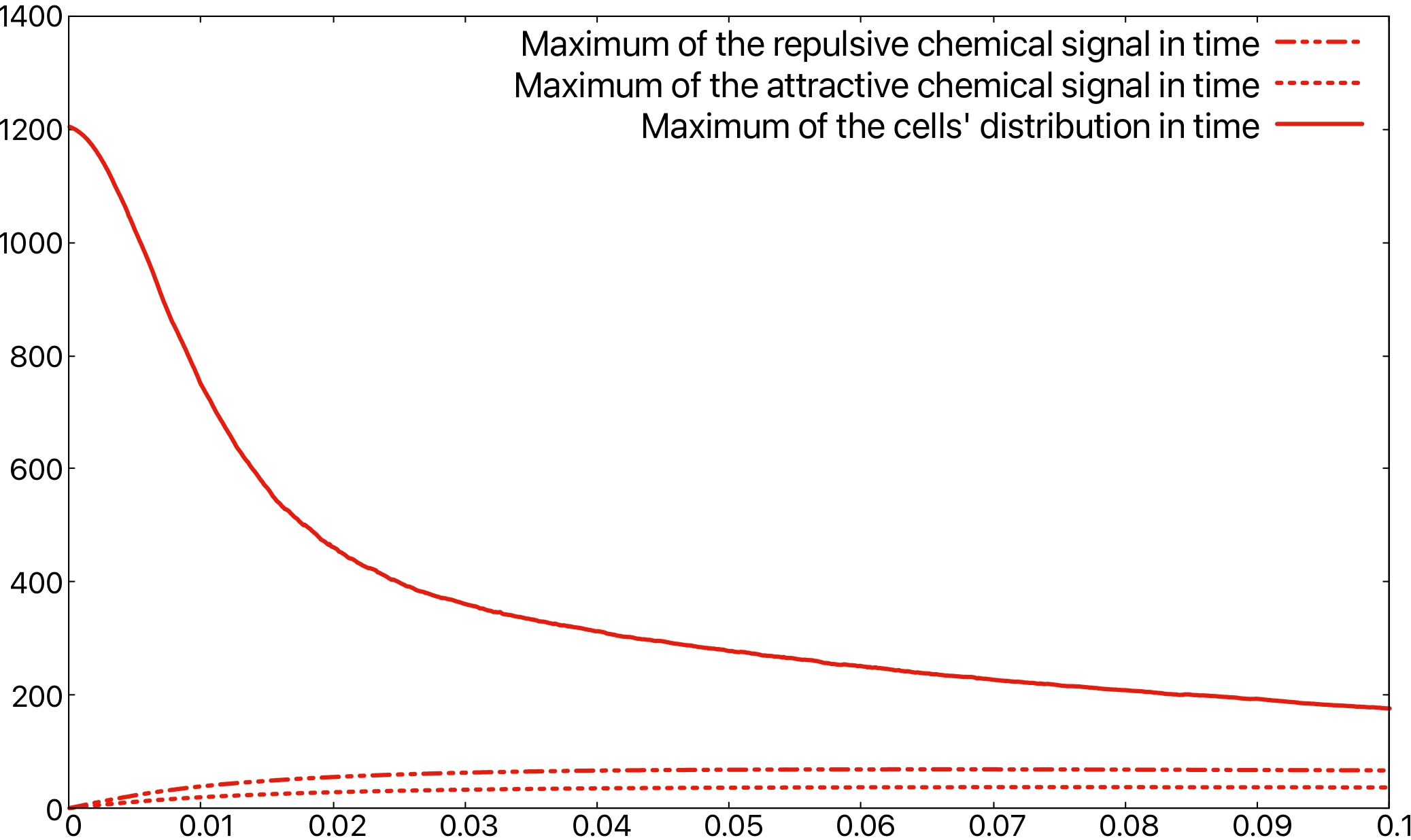}
}
\caption{{\bf{Global boundedness case.}} Simulations of model \eqref{problem} for $\tau=0$ (blue color) and $\tau=1$ (red color). Evolution of the cells' density and the chemical signals. Whilst for $\tau=1$, we have that $\max v_0$ and $\max w_0$ are given, for $\tau=0$ they are computed accordingly to what said in Algorithm \ref{AlgorithNumericoEllitico}. The quantitative values are indicated on the axes.}\label{fig:BoundednessbothaCases}
\end{figure}
Numerical simulations of taxis-driven models are by no means an easy goal and addressing technical and computational details connected with this analysis is a very challenging issue for the researchers in this field. This is essentially due to the fact that solutions to such
systems may exhibit interesting mathematical properties (as for instance blow-up) which cannot be easily reproduced in a discrete framework; indeed, spurious oscillations may appear during the simulations and this might lead to inconsistent results.

In this section we numerically test the presented theoretical results by simulating
system \eqref{problem} in two dimensions, for both the elliptic and parabolic version. The simulations are used to detect specific data and values of the parameters  which delineate regions where different behaviors of solutions may manifest. We are, in particular, interested in two specific scenarios: solutions are bounded and their asymptotic behavior converges to a constant stationary solution, or they suffer from chemotactic blow-up in finite time.  

To this aim we introduce a Finite Element Method (FEM) space
discretization, defined on a triangular mesh of the domain $\Omega$, with first order piecewise polynomials. On the other hand, we consider an implicit Euler scheme with time discretization $\Delta t$, and the algorithm at each time step consists of two operations: First we compute $v$ and $w$ as solutions of the (uncoupled) nonlinear equations in \eqref{problem} modeling the productions of these signals, being $u$ explicitly given from the previous time step. Successively, the values of $u$ at any node of the mesh are updated by solving the chemotaxis equation, exactly using the gained quantities of the repulsive and attractive substances at those nodes. Explicitly we have this 
\begin{algorithm}\label{AlgorithNumericoEllitico}
Let $j$ indicate the time step. For $j=0,1,\ldots$, we distinguish the cases $\tau=0$ and $\tau=1$:
\begin{itemize}
\item [$\triangleright$] $\tau=0$. Given $u_j$, we compute $v_{j}$ and $w_{j}$. Then we obtain $u_{j+1}$, by using $(u_j,v_{j},w_{j})$. 
\item [$\triangleright$] $\tau=1$. Given $(u_j,v_j,w_j)$, we compute $v_{j+1}$ and $w_{j+1}$. Then we obtain $u_{j+1}$, by using $(u_j,v_{j+1},w_{j+1})$. 
\end{itemize}
As to the blow-up scenario, we fix this threshold criterion: 
\begin{itemize}
\item [$\triangleright$] If for some $j$ it holds that $\max u_j\geq  10^4$, then $T_{max}\approx j \Delta t.$ 
\end{itemize}
\end{algorithm}
The computations are developed  by means of the the open source library \textsf{FreeFem++} (see \cite{HechFreeFem}). 
\subsection{Comparing blow-up and boundedness scenarios for the cases $\tau=0$ and $\tau=1$}
We consider a mesh of the unit circle $\Omega=\{(x_1,x_2)\in \R^2 \mid x_1^2+x_2^2<81 \}$ composed by 43148 triangular elements, and a time discretization $\Delta t = 10^{-5}$. We take $\chi=\xi=1$, $f(u)=\alpha u^k$, $\gamma_0=\gamma_1$ and, thence, $g(u)=\gamma_0 (1+u)^l$, $\Theta_0=\chi \alpha -\xi \gamma_0=0$ and the Gaussian bell-shaped initial values
 \begin{equation*}
 \begin{cases}
 u_0(x_1,x_2)=15 e^{-(x_1^2+x_2^2)}(81-(x_1^2+x_2^2)) & \textrm{ if } \tau=0,\\
 u_0(x_1,x_2)=15 e^{-(x_1^2+x_2^2)}(81-(x_1^2+x_2^2)),  v_0(x_1,x_2)= w_0(x_1,x_2)=  e^{-(x_1^2+x_2^2)} & \textrm{ if } \tau=1.
 \end{cases}
 \end{equation*}
Starting by highly spiky initial cells’ density (whose value is $1215$), our objective is finding production rates for the chemorepellent and chemoattractant so that either the value of such a spike becomes higher and the appearance of $\delta$-formations is observed, or decreases and stabilizes. As to the first goal (for which we adopted an adaptive mesh refinement), we choose $k,l$ outside the boundedness ranges in Theorems \ref{MainTheorem} and \ref{Main1Theorem}, and precisely $k=1.1$ and $l=0.9$. For both models ($\tau=0$ and $\tau=1$) it is seen (Figure \ref{fig:Blow-UPbothaCases}) that the cell distribution uncontrollably increases, up to an explosion at finite time. Nevertheless, a significant difference between the corresponding values of the blow-up time is seen, and this difference can be well appreciated in Subfigure \ref{SubFigCompar}.

Oppositely, without changing $u_0$, we can prevent the previous singularity by suitably changing the values of $k$ and $l$. Specifically, for $k=1.1$ and $l=1.2$ the model does not present any gathering effect and the cells and chemical densities remain bounded in time. We point out that the result obtained with these values of $k$ and $l$ is consistent with assumption (\ref{itemkminltheoremelliptic}) in Theorem \ref{MainTheorem}, but it cannot be discussed in the frame of Theorem \ref{Main1Theorem}. Finally, Figure \ref{fig:BoundednessbothaCases} quantifies the maximum values of $u, v$ and $w$; in particular, it is possible to observe how the trend of the curves representing $\max u$, $\max v$ and $\max w$ approach the plateau and they become constant throughout the time. Actually (data not shown) they eventually achieve the constant equilibrium $(u,v,w)=\left(\frac{m}{|\Omega|},\frac{\alpha}{\beta}\left(\frac{m}{|\Omega|}\right)^k,\frac{\gamma}{\delta}\left(1+\frac{m}{|\Omega|}\right)^l\right)$, which manifestly solves the stationary problem 
\begin{equation*}
\begin{cases}
0= \Delta u - \chi \nabla \cdot (u \nabla v)+\xi \nabla \cdot (u \nabla w)=\Delta v-\beta v +\alpha u^k=\Delta w - \delta w + \gamma_0 (1+u)^l  & \text{ in } \Omega \times (0,\infty),\\
u_{\nu}=v_{\nu}=w_{\nu}=0 & \text{ on } \partial \Omega \times (0,\infty).\\
\end{cases}
\end{equation*}
\subsection{Analyzing the influence of the parameters $k$ and $l$ on $\TM$}
With $\Omega$, $\Delta t$, $f(u)$, $g(u)$ and $\Theta_0$ introduced above, in Table \ref{TableInfluencekAndL} we schematize heuristic analyses of some simulations by  particularly indicating the value of the blow-up time for different initial configurations and some parameter sweep for $k,l,\Theta_0$. The first row of the table was inserted to show the consistency of the numerical simulations with the theoretical conclusion of Theorem \ref{MainTheorem}, when item (\ref{itemkUGUALEltheoremelliptic}) is considered. The other situations are of natural interpretation. The second and third rows give hints about the increasing of the blow-up time when $k$ decreases; in particular, for $k>l$, but $k<1\left(=\frac{2}{2}\right)$, gathering effects are appreciable. (Recall that for model \eqref{problemConFunzioniTempo}, blow-up in planar domain is ensured for $k>l$ and $k>1$.) If, indeed, we observe the part of the table dealing with the parabolic case, from the one hand boundedness is achieved for $k=l=0.5$, accordingly to the assumption (\ref{Item1-Parabol}) of Theorem \ref{Main1Theorem}, but from the other hand the same might occur for higher values.   
A further conclusion can be extrapolated by the fourth and sixth rows: indeed, one can observe that whenever the initial distribution of the chemoattractant $v_0$ presents a higher maximum, the related blow-up time decreases. 
\begin{table}
\begin{center}
\begin{tabular}{c|c|c|c|c|c|c||c}

                          & $u_0$                          & $v_0$             & $w_0$            & $k$         & $l$         & $\Theta_0$      & $\TM$    \\ \hline \hline
\multirow{3}{*}{$\tau=0$} & $15e^{-(x_1^2+x_2^2)}(81-x_1^2-x_2^2)$ &                   &                  & $0.5$ & $0.5$ & $0.36$  & $+\infty$       \\ \cline{2-8} 
                          & $15e^{-(x_1^2+x_2^2)}(81-x_1^2-x_2^2)$ &                   &                  & $1.2$ & $1$   & $-0.2$  & $0.00035$ \\ \cline{2-8} 
                          & $15e^{-(x_1^2+x_2^2)}(81-x_1^2-x_2^2)$ &                   &                  & $0.8$ & $0.6$ & $-0.64$  & $0.0188$       \\ \hline 
\multirow{4}{*}{$\tau=1$} & $15e^{-(x_1^2+x_2^2)}(81-x_1^2-x_2^2)$ & $e^{-(x_1^2+x_2^2)}$  & $e^{-(x_1^2+x_2^2)}$ & $1$   & $0.8$ & $-0.2$  & $0.025$  \\ \cline{2-8} 
                          & $15e^{-(x_1^2+x_2^2)}(81-x_1^2-x_2^2)$ & $e^{-(x_1^2+x_2^2)}$ & $e^{-(x_1^2+x_2^2)}$ & $0.5$   & $0.5$ & $0.52$  & $+\infty$   \\ \cline{2-8} 
                          & $15e^{-(x_1^2+x_2^2)}(81-x_1^2-x_2^2)$ & $5e^{-(x_1^2+x_2^2)}$ & $e^{-(x_1^2+x_2^2)}$ & $1$   & $0.8$ & $-0.2$  & $0.022$   \\
                           \cline{2-8} 
                          & $15e^{-(x_1^2+x_2^2)}(81-x_1^2-x_2^2)$ & $5e^{-(x_1^2+x_2^2)}$ & $e^{-(x^2+x_2^2)}$ & $0.8$ & $0.8$ & $-0.58$ & $+\infty$       \\ \hline
\end{tabular}
\end{center}
\caption{Approximated values of $\TM$ in terms of the sizes of $k,l, \Theta_0$ and the initial data.  With $\TM=+\infty$ we mean that the solution is globally bounded.}\label{TableInfluencekAndL}
\end{table}

\subsubsection*{Acknowledgments}
SF and GV are members of the Gruppo Nazionale per l'Analisi Matematica, la Probabilit\`a e le loro Applicazioni (GNAMPA) of the Istituto Na\-zio\-na\-le di Alta Matematica (INdAM). The authors are partially supported by the research projects \textit{Evolutive and stationary Partial Differential Equations with a focus on biomathematics}, funded by Fondazione di Sardegna (2019), and GV also by MIUR (Italian Ministry of Education, University and Research) Prin 2017 \textit{Nonlinear Differential Problems via Variational, Topological and Set-valued Methods} (Grant Number: 2017AYM8XW). 

\end{document}